\documentclass[11pt]{amsart}

\usepackage{amssymb,amsmath,enumerate,epsfig}
\usepackage{amsfonts,color}
\usepackage{a4,latexsym}
\usepackage{parskip}

\usepackage{hyperref}

\newtheoremstyle{Beweis}
  {0cm}                   
  {1cm}                   
  {\normalfont}           
  {}                      
  {\normalfont\itshape}   
  {.}                     
  { }                     
  {}                      

\newtheorem{proposition}{Proposition}[section]

\newtheorem{lemma}[proposition]{Lemma}
\newtheorem{corollary}[proposition]{Corollary}
\newtheorem{theorem}[proposition]{Theorem}
\newtheorem*{main}{Theorem}

\theoremstyle{definition}

\newtheorem{remark}[proposition]{Remark}
\newtheorem{remarks}[proposition]{Remarks}

\newtheorem{definition}[proposition]{Definition}
\theoremstyle{Beweis}

\newcommand{\longra}{\longrightarrow}
\newcommand{\Lra}{\Leftrightarrow}

\newcommand{\Ra}{\Rightarrow}
\newcommand{\Longlra}{\Longleftrightarrow}

\newcommand{\Longra}{\Longrightarrow}

\newcommand{\inv}{^{-1}}
\newcommand{\hermitian}[1]{^{H}\!{#1}}
\newcommand{\transpose}[1]{^{T}\!{#1}}

\newcommand{\NN}{\mathbb N}
\newcommand{\ZZ}{\mathbb Z}
\newcommand{\QQ}{\mathbb Q}
\newcommand{\RR}{\mathbb R}
\newcommand{\CC}{\mathbb C}

\newcommand{\PP}{\mathbb P}
\newcommand{\VV}{\mathbb V}
\newcommand{\WW}{\mathbb W}

\newcommand{\sO}{{\mathcal O}}

\newcommand{\sV}{{\mathcal V}}

\newcommand{\rk}{\operatorname{rk}}

\newcommand{\Aut}{\operatorname{Aut}}

\newcommand{\GL}{\operatorname{GL}}

\newcommand{\Hom}{\operatorname{Hom}}

\newcommand{\Centre}{\operatorname{centre}}

\newcommand{\ord}{\operatorname{ord}}
\newcommand{\red}{\operatorname{red}}
\newcommand{\lcm}{\operatorname{lcm}}

\newcommand{\Zdunits}{\left(\ZZ / d\ZZ\right)^*}

\newcommand{\Zd}{\ZZ / d\ZZ}
\newcommand{\CHn}{\CC H^n}

\newcommand{\qnf}[1]{{\QQ(\sqrt{ #1 })}}
\newcommand{\kronecker}[2]{{\left(\frac{#1}{#2}\right)}}
\newcommand{\latt}[1]{{\langle{#1}\rangle}}

\renewcommand{\bar}[1]{\overline{#1}}

\newcommand{\isom}{\cong}

\usepackage{color}
\newcommand{\kommentar}[1]{}

\begin{document}
\title{Singularities of ball quotients}

\author{Niko Behrens}
\address{Institut f\"ur Algebraische Geometrie, Leibniz Universit\"at
  Hannover, Welfengarten 1, 30167 Hannover, Germany} 
\email{behrens@math.uni-hannover.de}



%

\date{\today}

\maketitle

\begin{abstract}
  We prove a result on the singularities of ball quotients $\Gamma\backslash\CHn$. More precisely, we show that a ball quotient has canonical singularities under certain restrictions on the dimension $n$ and the underlying lattice. We also extend this result to the toroidal compactification $(\Gamma\backslash\CHn)^*$.
\end{abstract}

\section{Introduction}
Modular varieties are much studied objects in algebraic geometry. An example is the moduli space of polarised K3 surfaces which is a modular variety of orthogonal type. Similar modular varieties also occur in the context of irreducible symplectic manifolds. V.A.~Gritsenko, K.~Hulek and G.K.~Sankaran proved that the compactified moduli space of polarised K3 surfaces of degree $2d$ has canonical singularities. This result was used to show that this moduli space is of general type if $d>61$ (cf. \cite{MR2336040}).

In this paper we shall consider ball quotients  $\Gamma\backslash\CHn$ where $\Gamma$ is an arithmetic subgroup of the group of unitary transformations of a hermitian lattice $\Lambda$ of signature $(n,1)$, where $\Lambda\isom\sO^{n+1}$ for $\sO$ the ring of integers of some number field $\qnf{D},\ D<0$.

Varieties that arise in such a way  often also have an interpretation as moduli spaces. Examples appear in D.~Allcock's work \cite{MR1949641} as the moduli space of cubic threefolds and in the work of D.~Allcock, J.~Carlson and D.~Toledo \cite{MR1910264} as the moduli space of cubic surfaces. There are also papers of S.~Kond{\=o} \cite{MR1780433,MR2306153} on the moduli space of ordered 5 points on $\PP^1$ which appears as a two dimensional ball quotient, or on the moduli space of plane quartic curves which is birational to a quotient of a $6$-dimensional complex ball.

Ball quotient surfaces $B^2_\CC/\Gamma$ were studied by R.-P.~Holzapfel. Among  other  things he calculated formulae for the Euler number $e(\overline{B^2_\CC/\Gamma})$ and  the index $\tau(\overline{B^2_\CC/\Gamma})$ for a smooth model of the  Baily-Borel compactification and studied arithmetic aspects of ball quotient surfaces, e.g. \cite{MR653917, MR1685419}.

We give an outline of the organisation and the results of this paper.

In section \ref{SectionDefofObjects} we will recall the general definitions and basic properties of the objects that will be studied in the following sections. Furthermore the problem will be reduced to the local study of the action of the stabiliser subgroup $G$ on the tangent space $T_{[\omega]}\CHn$. 

In section \ref{SectionInterior} we will provide a criterion which implies that $\Gamma\backslash\CHn$ has canonical singularities, using methods similar to those of \cite{MR2336040}. This will be achieved by studying the representations of the action of $G$ on the tangent space $\Hom(\WW,\CC^{n+1}/\WW)$ and applying the Reid-Tai criterion. For this one first studies elements that do not act as quasi-reflections and then reduces the general situation to previous results.

For ball quotients $\Gamma\backslash\CHn$ the existence of toroidal compactifications  $(\Gamma\backslash\CHn)^*$ follows from the general theory described in  \cite{MR0457437}. Since all cusps in the Baily-Borel compactification are $0$-dimensional those toroidal compactifications are unique. We will describe them in section \ref{SectionBoundary}. We can then apply the results of section \ref{SectionInterior} and prove the main result:
\begin{main}
  The projective variety $(\Gamma\backslash\CHn)^*$ has canonical singularities for $n\geq 13$ provided  the discriminant of the number field $\qnf{D}$ associated to the lattice is not equal to $-3$,$-4$ or $-8$.
\end{main}
As in the orthogonal case this result can be used in the study of the Kodaira dimension of unitary modular varieties. This is the motivation of our work.

\subsection*{Acknowledgements}
This article is based on my PhD thesis. I want to thank my advisor K. Hulek for his guidance  and support. I would also like to thank V.A.~Gritsenko and G.K.~Sankaran for various helpful discussions.

\section{First definitions and properties}\label{SectionDefofObjects}
We first state the objects that we will study in the following sections.
Let $\qnf{D}$,  be an imaginary quadratic number field, i.e. $D<0$ a squarefree integer, and ${\sO}=\sO_\qnf{D}$ the corresponding ring of integers. Let $\Lambda$ be an $\sO$-lattice of signature $(n,1)$, i.e. a free $\sO$-module of rank $n+1$ with a hermitian form of signature $(n,1)$. Therefore we have an isomorphism $\Lambda\isom\sO^{n,1}$. The hermitian form given by this lattice will be denoted by $h(\cdot,\cdot)$.
When we fix a basis we get the isomorphism 
$$\psi:\ \Lambda\otimes_\sO\CC\isom\CC^{n,1}.$$
 The form induced by $\psi$ will also be denoted by $h(\cdot,\cdot).$

Starting with the lattice $\Lambda$ we define the $n$-dimensional complex hyperbolic space as
\begin{eqnarray}
  \CHn:=\{[\omega]\in\PP(\Lambda\otimes_\sO\CC);\ h(\omega,\omega)<0\}.
\end{eqnarray}
By definition $\CHn$ has a natural underlying lattice structure given by $\Lambda$. 
Due to G. Shimura there is the identification $\CHn\isom U(n,1)/(U(n)\times U(1))$, cf. \cite{MR0156001}.

For future use we define 
\begin{eqnarray}
  U(\Lambda):=\text{group of automorphisms of}\ \Lambda.
\end{eqnarray}
After choosing a suitable basis  $U(\Lambda)_\CC:=U(\Lambda)\otimes_\sO \CC\isom U(n,1)$.
Now let $\Gamma<U(\Lambda)$ be a subgroup of finite index.
We denote  the {\em $n$-dimensional ball quotient} by
\begin{eqnarray}
  \Gamma\backslash\CHn.
\end{eqnarray}
This ball quotient is a quasi-projective variety by \cite{MR0216035}. It can be compactified using toroidal compactification which gives rise to a unique projective variety $(\Gamma\backslash\CHn)^*$.

One can give a description of the ramification divisors. For this purpose let 
\begin{eqnarray}
  f_\Gamma:\ \CHn\longra \Gamma\backslash\CHn\label{mapf_Gamma}
\end{eqnarray}
be the quotient map. The elements fixing a divisor in $\CHn$ are the quasi-reflections. Thus the ramification divisors of $f_\Gamma$ are the fixed loci of elements of $\Gamma$ acting as quasi-reflections.

As \cite{MR2336040} and \cite{MR1255698} did for the K3 (orthogonal) case we will investigate the local action on the tangent space. Fix a point $[\omega]\in\CHn$ and define the stabiliser of $[\omega]$:
\begin{eqnarray}
  G:=\Gamma_{[\omega]}:=\{g\in\Gamma;\ g[\omega]=[\omega]\}.
\end{eqnarray}
This group is finite by results of \cite[4.1.2]{MR1685419} or \cite[pp. 1]{MR0314766}. Next define for $\omega\in\Lambda\otimes_\sO\CC$ the line $\WW:=\CC\omega$ corresponding to $[\omega]$. Then we can define the following sublattices of the lattice $\Lambda$:
\begin{definition}
  \begin{eqnarray*}
    S:=\WW^\perp\cap \Lambda,\ T:=S^\perp\cap\Lambda,
  \end{eqnarray*}
where the orthogonal complements are taken with respect to the form $h(\cdot,\cdot)$.
\end{definition}
As before we can complexify these lattices. We denote the resulting vector spaces by
$$S_\CC:=S\otimes_\sO\CC,\ T_\CC:=T\otimes_\sO\CC.$$
Now we have to study some properties of these lattices. Some proofs will be similar to those in \cite[2.1]{MR2336040}.
\begin{lemma}\label{LemmaScapT=0}
  $S_\CC\cap T_\CC=\{0\}$.
\end{lemma}
\begin{proof}
  Let $x\in S_\CC\cap T_\CC$. Then $h(x,x)=0$, since $x\in T_\CC=S_\CC^\perp$. Therefore it suffices to show that $h(\cdot,\cdot)$ is positive definite on $S_\CC$. Consider $\WW\subset\Lambda_\CC$ with $\CC$-basis $\{\omega\}$. Hence  $h(\omega,\omega)<0$ as $[\omega]\in\CHn$. Hence the hermitian form has signature $(0,1)$ on $\WW$ and thus signature $(n,0)$ on $\WW^\perp$. By definition $S_\CC\subset \WW^\perp$ and the result follows.
\end{proof}
To describe the singularities we will study the action of the stabiliser $G$ on the tangent space $T_{[\omega]}\CHn$. Therefore we need a more concrete description.
\begin{remark}
The tangent space is known for the Grassmannian variety $G(1,n+1)$, e.g. \cite[Chapter II, \S2]{MR770932}. Hence we get $T_{[\omega]}\CHn=\Hom(\WW,\CC^{n+1}/\WW)$.
\end{remark}
From now on we denote this tangent space by $V:=\Hom(\WW,\CC^{n+1}/\WW)$ and investigate the quotient $G\backslash V$ in more detail.

\subsection{Representations of cyclic groups over quadratic number fields}
Before we state first results we have to study  the beviour of representations of the cyclic group $\Zd$ over a given quadratic number field for an integer $d>1$. We denote the $d$th cyclotomic polynomial by $\phi_d$. 
The classification of irreducible representations will depend on whether $\phi_d$ is irreducible over $\qnf{D}$ or not. For the following we denote by $\left(\frac{\cdot}{\cdot}\right)$ the {\em Kronecker symbol}.
\begin{proposition}\label{PropositionirredRepsoverqnf}
  Let $\rho:\ \Zd\longra\Aut(W)$ be a representation of $\Zd$ on the $d$-dimensional vector space $W$ over $\qnf{D}$. Then
\begin{itemize}
  \item[(i)] there is a unique irreducible faithful representation $V_d$ if $\phi_d$ is irreducible. The eigenvalues of $\rho|_{V_d}(\zeta_d)$ are the primitive $d$th roots of unity.
\item[(ii)] there are two irreducible faithful representations $V_d',V_d''$ if $\phi_d$ is reducible. The eigenvalues of $\rho|_{V_d'}(\zeta_d)$ are the primitive $d$th roots of unity $\zeta_d^a$ with $\kronecker{D}{a}=1$ for $a\in\Zdunits$. The eigenvalues of $\rho|_{V_d''}(\zeta_d)$ are $\zeta_d^a$ for the remaining $a\in\Zdunits$, i.e. the $a$ with $\kronecker{D}{a}=-1$.
\end{itemize}
\end{proposition}
\begin{proof}
  This follows from results of L. Weisner \cite{MR1502846} and standard calculations in representation theory.
\end{proof}
\begin{remark}
For specific $D$ we can rephrase this as follows:
 \begin{itemize}
\item[(a)] If $D>0$, then for each eigenvalue $\zeta_d^a$ of $\rho|_{V_d'}(\zeta_d)$ the complex conjugate $\zeta_d^{d-a}$ is an eigenvalue as well.
\item[(b)] If $D<0$, then for each eigenvalue $\zeta_d^a$ of $\rho|_{V_d'}(\zeta_d)$ the complex conjugate $\zeta_d^{d-a}$ is an eigenvalue of $\rho|_{V_d''}(\zeta_d)$.
\end{itemize}
\end{remark}
As the Proposition shows we have two different irreducible representations in case (ii).
\begin{remarks}
  Let $d$ be a positive integer.
\begin{itemize}
  \item[(i)] Sometimes we do not want to specify which of the representations $V_d$ resp. $V_d'$ or $V_d''$ we take and only refer to $\sV_d$ which will denote the  appropriate representation in the given situation.
\item[(ii)] We define 
$\VV_d:=\sV_d\otimes_\qnf{D}\CC$.
\end{itemize}
\end{remarks}

\section{The interior}\label{SectionInterior}
In this section we study the decomposition of $\qnf{D}$-vector spaces associated to  the lattices $S$ and $T$ under the action of a cyclic group. This enables us to state  results on canonical singularities for ball quotients $\Gamma\backslash\CHn$.
\begin{lemma}
  $G$ acts on $S$ and $T$.
\end{lemma}
\begin{proof}
  $G$ acts on $\WW$ and on $\Lambda$, hence on $S=\WW^\perp\cap \Lambda$ and on $T=S^\perp\cap\Lambda$.
\end{proof}
\begin{lemma}
  The spaces $S_\CC$ and $T_\CC$ are $G$-invariant subspaces of the vector space $\Lambda_\CC$.
\end{lemma}
\begin{proof}
  We will only  give a proof for $S_\CC$ as $T_\CC$ is similar. Let $x\in T_\CC$, $y\in S_\CC$, $\omega\in\WW$ and $g\in G$. Then 
\begin{eqnarray*}
  0=h(y,\omega)=h(g(y),g(\omega))=\overline{\alpha(g)}\cdot h(g(y),\omega).
\end{eqnarray*}
As $\alpha (g)\neq 0$ we get $h(g(y),\omega)=0$, i.e. $g(y)\in S_\CC$.
\end{proof}
The group $G$ has been defined as the stabiliser of $[\omega]$ and therefore the equation
$$g(\omega)=\alpha(g)\omega$$
holds for all $g\in G$, where 
$$\alpha:\ G\longra\CC^*$$
is a group homomorphism. Denote its kernel by
$  G_0:=\ker \alpha$.
Analogous to the previous define
$$S_\qnf{D}:=S\otimes_\sO \qnf{D}\ \text{and}\ T_\qnf{D}:=T\otimes_\sO\qnf{D}. $$
\begin{lemma}
  The group $G_0$ acts trivally on $T_\qnf{D}$.
\end{lemma}
\begin{proof}
  Let $x\in T_\qnf{D}$ and $g\in G_0$. Then
$$  h(\omega,x)=h(g(\omega),g(x))=h(\omega,g(x))$$
and $x-g(x)\in\WW^\perp\cap\Lambda_\qnf{D}=S_\qnf{D}$. Therefore the result follows by Lemma \ref{LemmaScapT=0}.
\end{proof}
The quotient $G/G_0$ is a subgroup of $\Aut\WW\isom\CC^*$ and therefore cyclic. The order of this group will be denoted by $r_\omega:=\ord(G/G_0)$.

\begin{lemma}
   The space $T_\qnf{D}$ decomposes as a $G/G_0$-module
\begin{itemize}\label{LemmaDecompofTasGmodG0module}
  \item[(i)] into a direct sum of $V_{r_\omega}$'s, i.e. $\varphi(r_\omega)$ divides $\dim T_\qnf{D}$, if $V_{r_\omega}$ is irreducible over $\qnf{D}$,
\item[(ii)] into a direct sum of $V_{r_\omega}'$'s and $V_{r_\omega}''$'s, in particular $\frac{\varphi(r_\omega)}{2}$ divides $\dim T_\qnf{D}$, if there exist a decomposition $V_{r_\omega}=V_{r_\omega}'\oplus V_{r_\omega}''$ over $\qnf{D}$.
\end{itemize}
\end{lemma}
\begin{proof}
   It remains to show, that the only element having $1$ as an eigenvalue on $T_\CC$ is identity element in $G/G_0$. This suffices as $G/G_0\isom\mu_{r_\omega}$ and by the Chinese Remainder Theorem $(\ZZ/r_\omega\ZZ)^*\isom ((\ZZ/p_1\ZZ)^*)^{a_1}\times\dots\times((\ZZ/p_t\ZZ)^*)^{a_t}$ for suitable $p_i$ and $a_i$. 
Assume that $g\in G-G_0$ with $g(x)=x$ for a $x\in T_\CC$. Then
$$h(\omega,x)=h(g(\omega),g(x))=\alpha(g)\cdot h(\omega,x).$$
As $a(g)\neq 1$  we get $h(\omega,x)=0$ and therefore $x=0$.
\end{proof}
\begin{corollary}
   For $g\in G$ the space $T_\qnf{D}$ decomposes as a $g$-module into a direct sum of $V_r$'s resp. $V_{r}'$'s or $V_{r}''$'s of dimension $\varphi(r)$ resp. $\frac{\varphi(r)}{2}$.
\end{corollary}
\begin{proof}
  Similar to Lemma \ref{LemmaDecompofTasGmodG0module}.
\end{proof}

\subsection{Reid-Tai criterion}
Let $M=\CC^k$, $A\in \GL(M)$ be of order $l$ and fix a primitive $l$th root of unity $\zeta$. Consider the eigenvalues $\zeta^{a_1},\dots,\zeta^{a_k}$ of $A$ on $M$, where $0\leq a_i<l$. Define the {\em Reid-Tai sum} of $A$ as
$$\Sigma(A):=\sum_{i=1}^k \frac{a_i}{l}.$$
\begin{theorem}[Reid-Tai criterion]
  Let $H$ be a finite subgroup of $\GL(M)$ without quasi-reflections. Then $M/H$ has canonical singularities if and only if 
$$\Sigma(A)\geq1$$
for every $A\in H$, $A\not=I$.
\end{theorem}
\begin{proof}
  \cite[(4.11)]{MR927963} and \cite[Theorem 3.3]{MR669424}.
\end{proof}

\subsection{Singularities in the interior}
Now we will apply the Reid-Tai criterion to $G\backslash V$. But as the techniques used do not work for general $D$ we will restrict ourself to the case $D<D_0$, where $D_0:=-3$. For a rational number $q$ we denote the {\em fractional part} of $q$ by $\{q\}$. 
Remember that the eigenvalue of $g$ on $\WW$ is $\alpha(g)$, where the order of $\alpha(g)$ is $r$. First we will give a bound on $r$.

\begin{lemma}\label{LemmagnoQRphi(r)geq10}
  Suppose $g$ does not act as a quasi-reflection on $V$ and $D<0$. Then the Reid-Tai sum satisfies $\Sigma(g)\geq1$, if $\varphi(r)\geq10$. If we assume additionally $D<D_0$ then this holds even for $\varphi(r)=4$.
\end{lemma}
\begin{proof}
  We denote the copy of $V_r\otimes\CC$ resp. $V_r'\otimes\CC$ or $V_r''\otimes\CC$ that contains $\omega$ by $\VV_r^\omega$. Now choose a primitive $m$th root of unity $\zeta$ and let $0<k_i<r$ be the $\varphi(r)$ integers coprime to $r$. We consider the eigenvalue $\alpha(g)=\zeta^{\frac{mk_1}{r}}$ of $g$ on $\WW$, i.e. on $\WW^\vee$ this leads to the eigenvalue $\overline{\alpha(g)}=:\zeta^{\frac{mk_2}{r}}$. Additionally we have to consider the eigenvalues of $g$ on $V_r^\omega\cap\CC^{n+1}/\WW$, where the definition of $V_r^\omega$ is similar to the one of $\VV_r^\omega$. Thus on this space we will have eigenvalues $\zeta^{\frac{mk_i}{r}}$ for some $k_i\in A-\{k_1\}$. Here $A=A_D^r$ is a subset of $\{k_1,\dots,k_{\varphi(r)}\}$ depending on the decomposition behavior of $\phi_r$ with $k_1\in A$ and $\# A=\varphi(r)$ resp. $\#A=\frac{\varphi(r)}{2}$ (cf. Proposition \ref{PropositionirredRepsoverqnf}).\\
Now consider the eigenvalues of $g$ on $\Hom(\WW,\VV_r^\omega\cap\CC^{n+1}/\WW)$, which are $\zeta^\frac{mk_2}{r}\zeta^\frac{mk_i}{r}$ for $k_i\in A-\{k_1\}$. Thus $\Sigma(g)\geq \sum_{k_i\in A-\{k_1\}} \left\{\frac{k_2+k_i}{r}\right\}$.

First we want to show that there are only finitely many choices for $r$ that lead to a contribution less than $1$ to the Reid-Tai sum. Therefore look at the estimate $\sum_{k_i\in A-\{k_1\}} \left\{\frac{k_2+k_i}{r}\right\}\geq \sum _{j=1}^{\frac{\varphi(r)}{2}-1} \frac{j}{r}$ and study the prime decomposition $r=p_1^{a_1}\cdots p_s^{a_s}$, where we assume $p_i<p_j$ for $i<j$. It is now easy to show that this sum contributes at least $1$ to $\Sigma(g)$, unless we are in one of the following cases\\
\begin{minipage}{5cm}
  \begin{center}
  $r=2^a\cdot p^b\cdot q$
\end{center}
\begin{center}
\begin{tabular}{|c|c||c|c|}
\hline
a&b&p&q\\
\hline\hline
$<$3&1&3&$<$11\\
1&1&3&11\\
1&1&3&13\\
1&2&3&5\\
1&1&5&7\\
\hline
\end{tabular}
\end{center}
\end{minipage}
\begin{minipage}{4cm}
\begin{center}
   $r=p^aq^b$
\end{center}
 \begin{center}
\begin{tabular}{|c|c||c|c|}
\hline  
$a$& $b$& $p$& $q$\\ 
\hline\hline
1& 1& 2&$\leq$19\\
1& 1& 3&5,7\\
\hline
2&1&2&$<$11\\
2,3&2&2&3\\
1&2&2&5\\
\hline
3&1&2&$<$7\\
\hline
4&1&2&3\\
\hline
2&1&3&2\\
\hline
3&1&3&2\\
\hline
\end{tabular}
\end{center}
\end{minipage}
\begin{minipage}{4cm}
\begin{center}
   $r=p^a$
\end{center}
 \begin{center}
\begin{tabular}{|c||c|}
  \hline
a& p\\
\hline\hline
1&$<$11\\
\hline
2&3\\
\hline
$\leq5$&2\\
\hline
\end{tabular}
\end{center}
\end{minipage}
For these remaining values of $r$ we can calculate the contribution of $g$ on \linebreak $\Hom(\WW,\VV_r^\omega\cap\CC^{n+1}/\WW)$ to $\Sigma(g)$ in more detail as 
\begin{eqnarray*}
 \operatorname{mc}(r)&:=&\min_{\substack{D<0\\ \text{suitable}}}\min_{k_2\in A}  \sum_{\substack{k_i\in A-\{k_1\}\\ \kronecker{D}{k_i}=\kronecker{D}{r-k_2}}} \left\{\dfrac{k_2+k_i}{r}\right\}.
\end{eqnarray*}
By `suitable' we mean that we only consider number fields $\qnf{D}$ that lead to case (ii) in Proposition \ref{PropositionirredRepsoverqnf}. If there is no such number field we have to omit the first `$\min$' and the Kronecker symbol in the definition of $\operatorname{mc}(r)$. As there are only finitely many such number fields computer calculation yields $\operatorname{mc}(r)\geq1$ for $\varphi(r)\geq 10$ and $\varphi(r)=4$ if  we restrict to $D<D_0$.
\end{proof}
\begin{remark}\label{Remark_mcrforothervalues}
  The same calculations of $\operatorname{mc}(r)$ shows that $\Sigma(g)\geq1$ for $r=9,16,18$ and no restriction on $D<0$.
\end{remark}
\begin{lemma}\label{LemmagnoQRR=1,2}
  Assume that $g\in G$ does not act as a quasi-reflection on $V$. Additionally let $r=1,2$ and $D\neq -1,-2$. Then $\Sigma(g)\geq1$.
\end{lemma}
\begin{proof}
  As $r=1,2$ we have $\alpha(g)=\pm1$. 
  With an analogous statement as in \cite[Proposition 2.9]{MR2336040} we get that $g$ is not of order $2$ and $g^2$  acts trivially on $T_\CC$ but not on $S_\CC$. Therefore let $g$ act  on the subspace $\Hom(\WW,\VV_d)\subset V$ as $\pm\sV_d$ with $d>2$, for a representation $\sV_d$ from the decomposition of $S_\CC$ as a $g$-module over $\qnf{D}$. This contributes at least 
\begin{eqnarray*}
\min_{\substack{D<0\\\text{suitable}}}\min_{r\in\{1,2\}}\min_{\alpha=\pm1}\sum_{\substack{(k_i,d)=1\\ \kronecker{D}{k_i}=\alpha}} \left\{\dfrac{1}{r}+\dfrac{k_i}{d}\right\}\geq \sum_{j=1}^{\frac{\varphi(d)}{2}} \dfrac{j}{d}.\label{ProofinequalityHom(WW,VVd)for_r=1,2}
\end{eqnarray*}
to $\Sigma(g)$. As in the proof of Lemma \ref{LemmagnoQRphi(r)geq10} we have to modify this expression if $\sV_d=V_d$ and with analogous arguments we can reduce this to a question about a finite number of $d$'s. Computer calculation for these $d$ shows that they contribute at least $1$ unless $d=8$. But $\sV_8=V_8$ for $D\neq-1,-2$ and therefore we can choose complex conjugate eigenvalues.
\end{proof}
Now we can state a general result if $g$ is not a quasi-reflection.
\begin{theorem}\label{TheoremforgnoQRandngeq11}
  Suppose $g\in G$ does not act as a quasi-reflection on $V$. Then $\Sigma(g)\geq1$, if $D<D_0$ and $n\geq11$.
\end{theorem}
\begin{proof}
  Let $m$ be the order of $g$ and $\zeta$ be a primitive $m$th root of unity. On the space $\Hom(\WW,\VV_d)\subset V$ the element $g$ has eigenvalues $\zeta^{\frac{mc}{r}}\zeta^{\frac{mk_i}{d}}$ for fixed $0<c<r$ with $(c,r)=1$, and $k_i\in A$, where $A=A_D^d$ is defined as in the proof of Lemma \ref{LemmagnoQRphi(r)geq10}, so $\# A=\dim_\CC \VV_d$. Therefore the contribution of $g$ on this subspace is given by
$$\sum_{k_i\in A}\left\{\frac{c}{r}+\frac{k_i}{d}\right\}.$$
This is greater or equal to $\sum_{j=1}^{\frac{\varphi(d)}{2}}\frac{j}{d}$ for $d\not\in\varphi\inv(\{2,4,6,8\})$, and this contributes less than $1$ if 
\begin{eqnarray}
  d&=&1,2,\dots,10,12,14,15,16,18,20,22,24,26,28,30,36,\notag\\&&40,42,48,54,60,66,84,90\label{listofdofroughestimation}
\end{eqnarray}
(as in the proof of Lemma \ref{LemmagnoQRR=1,2}). We can calculate the contribution for each $d$ and $r$, but to simplify calculations define
\begin{eqnarray}
  c_{\operatorname{min}}(d)&:=&\min_{0\leq a<d}\sum_{\substack{0<b<d\\(b,d)=1}}\left\{\dfrac{b+a}{d}\right\},\ \text{resp.}\\
c_{\operatorname{min}}^{\red}(d)&:=&\min_{\substack{D<0\\ \text{suitable}}}\min_{\alpha=\pm1}\min_{0\leq a<d}\sum_{\substack{0<b<d\\(b,d)=1\\ \kronecker{D}{b}=\alpha}}\left\{\dfrac{b+a}{d}\right\}.
\end{eqnarray}
If there exist at least one imaginary quadratic number field for which the cyclotomic polynomial $\phi_d$ is reducible we have to calculate $c_{\operatorname{min}}^{\red}(d)$ for those $D$. If there exists no such $D$ we will use $c_{\operatorname{min}}(d)$. Both expressions only depend on $d$ and are a lower bound for the contribution to $\Sigma(g)$ as shown in \cite[Proof of Theorem 2.10]{MR2336040}. By computer calculation all $d$ except $d=1,2,3,4,6,7,8,12,14,15,20,24,30$ contribute at least $1$ to the Reid-Tai sum. For these $d$ we get 
\begin{eqnarray*}
&  c_{\min}^{\red}(30)=11/15,\ c_{\min}^{\red}(24)=5/6,\  c_{\min}^{\red}(20)=4/5,\   c_{\min}^{\red}(15)=11/15,&\\ &c_{\min}^{\red}(14)=4/7,\ c_{\min}^{\red}(12)=1/3,\  c_{\min}^{\red}(8)=1/4,&\\
& c_{\min}^{\red}(7)=4/7,\ c_{\min}^{\red}(6)=0,\ c_{\min}^{\red}(4)=0,c_{\min}^{\red}(3)=0.&
\end{eqnarray*}
As we know, $T_\CC$ decomposes into a direct sum of $\VV_r$, while we can assume the space $S_\CC$ decomposes into a direct sum of $\VV_d$ where $d\in\{1,2,3,4,6,7,8,12,14,$ $15,20,24,30\}$. In the following we only consider $D<D_0$ and a count of dimensions leads to the equation
\begin{eqnarray}
  \dim \VV_r\cdot \lambda +\nu_1+\nu_2+2\nu_3+2\nu_4+2\nu_6+\frac{6}{2}\nu_7+\frac{4}{2}\nu_8&&\notag\\
+\frac{4}{2}\nu_{12}+\frac{6}{2}\nu_{14}+\frac{8}{2}\nu_{15}+\frac{8}{2}\nu_{20}+\frac{8}{2}\nu_{24}+\frac{8}{2}\nu_{30}&=&n+1,\label{Equationdimensioncountrepresentations}
\end{eqnarray}
where $\lambda$ denotes the multiplicity of $\VV_r$ in $T_\CC$ and $\nu_d$ denotes the multiplicity of $\VV_d$ in $S_\CC$. Note that we can assume $\sV_d=V_d'$ or $V_d''$ for $d\in\{7,8,12,14,15,20,$ $24,30\}$. If  not it would contribute at least $1$ to $\Sigma(g)$, as shown in  \cite[Theorem 2.10]{MR2336040}. For the quotient $\Lambda_\CC/\VV_r^\omega$ we denote by $\nu_r$ the multiplicity of $\VV_r$ in $\Lambda_\CC / \VV_r^\omega$ as a $g$-module, where $\VV_r^\omega$ as before denotes the copy that contains $\omega$. Now we can calculate the (minimal) contribution of $\Hom(\WW,\VV_d)$ to $\Sigma(g)$ as
\begin{eqnarray}
  \sum_{(a,d)=1}\left\{\dfrac{a}{d}+\dfrac{k_1}{r}\right\}\ \text{resp.}\ \min_{ D<D_0}\min_{\alpha=\pm1}\sum_{\substack{(a,d)=1\\ \kronecker{D}{a}=\alpha}}\left\{\dfrac{a}{d}+\dfrac{k_1}{r}\right\}.
\end{eqnarray}
According to Lemma \ref{LemmagnoQRphi(r)geq10} and Remark \ref{Remark_mcrforothervalues} we have to investigate the cases $r\in\{3,4,6\}=\varphi\inv(2)$, $r\in\{7,14\}\subset\varphi\inv(6)$ and $r\in\{15,20,24,30\}\subset\varphi\inv(8)$. Now we have to study different cases:
\begin{itemize}
  \item[(1)] Let $\varphi(r)=2$. The contributions of the $\VV_d$ with $\varphi(d)\geq4$ are greater or equal to $1$ and (\ref{Equationdimensioncountrepresentations}) becomes
\begin{eqnarray}
  \nu_1+\nu_2+2\nu_3+2\nu_4+2\nu_6=n+1-2=n-1.\notag
\end{eqnarray}
For the $6$ possible cases of the choice of $(r,k_1)$, namely $r\in\{3,4,6\}$ and $k_1\in\{1,r-1\}$, the other contributions are at least
\begin{center}
\begin{tabular}{|c|c|}
  \hline
$d$ & contribution\\
\hline\hline
1 &1/6 \\
\hline
2 & 1/6\\
\hline
3 & 1/3\\
\hline
4 & 1/2\\
\hline
6 & 1/3\\
\hline
\end{tabular}
\end{center}
In all cases we see $\Sigma(g)\geq1$ if $n-1\geq6$.
\item[(2)] Let $r=7,14$. We can assume $D=-7$ since if this is not the case explicit calculations show that$\VV_r^\omega$ will contribute at least $1$ to $\Sigma(g)$. Equation (\ref{Equationdimensioncountrepresentations}) becomes
\begin{eqnarray}
  \nu_1+\nu_2+2\nu_3+2\nu_4+2\nu_6+3\nu_7+3\nu_{14}=n-2\notag
\end{eqnarray}
and the contributions are
\begin{center}
\begin{tabular}{|c|c|}
  \hline
$d$ & contribution\\
\hline\hline
1 &1/14 \\
\hline
2 & 1/14\\
\hline
3 & 3/7\\
\hline
4 & 4/7\\
\hline
6 & 3/7\\
\hline
7 & 4/7\\
\hline
14 & 4/7\\
\hline
\end{tabular}
\end{center}
and $4/7$ from $\VV_r^\omega$. So we may assume that $\nu_3=\nu_4=\nu_6=\nu_7=\nu_{14}=0$, because otherwise the contribution will be $\geq1$. So $\Sigma(g)\geq1$, if $\nu_1+\nu_2\geq6$ resp. $n\geq8$.
\item[(3)] Let $r=15,20,24,30$. Analogously to the last case we can assume that $D=-5,-6,-15$.
\begin{itemize}
  \item[(a)] Let $D=-5$. Hence we get the equation
\begin{eqnarray}
  \nu_1+\nu_2+2\nu_3+2\nu_4+2\nu_6+4\nu_{20}=n-3.\notag
\end{eqnarray}
The contributions are 
\begin{center}
\begin{tabular}{|c|c|}
  \hline
$d$ & contribution\\
\hline\hline
1 &1/30 \\
\hline
2 & 1/30\\
\hline
3 & 5/12\\
\hline
4 & 8/15\\
\hline
6 & 5/12\\
\hline
20 & 4/5\\
\hline
\end{tabular}
\end{center}
and $4/5$ from $\VV_r^\omega$. So $\Sigma(g)\geq1$ unless $\nu_1+\nu_2\leq5$ resp. $n\leq8$.
\item[(b)] Let $D=-6$, giving the equation
\begin{eqnarray}
  \nu_1+\nu_2+2\nu_3+2\nu_4+2\nu_6+4\nu_{24}=n-3.\notag
\end{eqnarray}
The contributions of $\VV_{24}$ and $\VV_r^\omega$ are $5/6$. So $\Sigma(g)\geq1$ unless $\nu_1+\nu_2\leq4$ resp. $n\leq7$.
\item[(c)] The last case is $D=-15$. So we get the equation
\begin{eqnarray}
  \nu_1+\nu_2+2\nu_3+2\nu_4+2\nu_6+4\nu_{15}+4\nu_{30}=n-3.\notag
\end{eqnarray}
The contributions of $\VV_{15}$, $\VV_{30}$ and $\VV_r^\omega$ are $11/15$. So $\Sigma(g)\geq1$, if $\nu_1+\nu_2\geq8$ resp. $n\geq11$.
\end{itemize}
\end{itemize}
\end{proof}
Now we can state a first result about the singularities of ball quotients.
\begin{corollary}
  Let $D<D_0$ and $n\geq11$, then $\Gamma\backslash\CHn$ has canonical singularities away from the branch divisors.
\end{corollary}
\begin{proof}
This directly follows from Theorem \ref{TheoremforgnoQRandngeq11}, the Reid-Tai criterion and the discussion for the map (\ref{mapf_Gamma}).
\end{proof}
For the rest of this section we will now study elements $h=g^k$ that act as quasi-reflections on the tangent space $V$. We will start to describe how $\Lambda_\qnf{D}$ decomposes as a $g$-module.
\begin{proposition}
   Let $h=g^k$ be a quasi-reflection on $V$ for $g\in G$ and $n\geq2$. As a $g$-module we have a decomposition of the form $$\Lambda_\qnf{D}\isom\sV_{m_0}\oplus\bigoplus_j\sV_{m_j}$$ for some $m_i\in\NN$. Then
\begin{enumerate}
  \item[(i)] $(m_0,k)=m_0$ and $2(m_j,k)=m_j$, or $2(m_0,k)=m_0$ and $(m_j,k)=m_j$ for $j\geq1$ in the cases $D<D_0$ and $D=-2$,
\item[(ii)] $(m_0,k)=m_0$ and $l(m_j,k)=m_j$, or $l(m_0,k)=m_0$ and $(m_j,k)=m_j$, $l\in\{2,4\}$, for $j\geq1$ in the case $D=-1$,
\item[(iii)] $(m_0,k)=m_0$ and $l(m_j,k)=m_j$, or $l(m_0,k)=m_0$ and $(m_j,k)=m_j$, $l\in\{2,3,6\}$, for $j\geq1$ in the case $D=-3$.
\end{enumerate}
\end{proposition}
\begin{proof}
  As a $g$-module $\Lambda_\qnf{D}$ decomposes into $\sV_r^\omega\oplus\bigoplus_{i}\sV_{d_i}$ for some $d_i\in\NN$. As $h$ is a quasi-reflection on $V$, all but one of the eigenvalues  on $V$ must be $1$. First fix an $i$. Now  define $\sV_d:=\sV_{d_i}$ and $d':=\frac{d}{(k,d)}$, then the eigenvalues of $h$ on $\sV_d$ are primitive $d'$th roots of unity of multiplicity $\frac{\dim \sV_d}{\dim \sV_{d'}}$. We want to give restrictions on the $d_i$:
\begin{enumerate}
  \item[(1)] $\dim \sV_{d'}\leq2$: Assume that the dimension is at least $3$. One can choose three distinct eigenvalues $\zeta, \zeta', \zeta''$ on $\sV_{d'}$, such that $h$ has eigenvalues $\alpha(h)\inv\zeta, \alpha(h)\inv\zeta'$ and $\alpha(h)\inv\zeta''$ on $V$ and at most one of these eigenvalues can be $1$.
\item[(2)] $\frac{\dim\sV_d}{\dim \sV_{d'}}=2\Ra \dim\sV_{d'}=1$: Assume $\dim \sV_{d'}\geq2$ under the given condition. Denote two of the $\dim\sV_{d'}$ eigenvalues of multiplicity $2$ of $h$ on $\sV_d$ by $\zeta,\zeta'$. So one would have the eigenvalues $\alpha(h)\inv\zeta$ and $\alpha(h)\inv\zeta'$ of multiplicity $2$ on $V$.
\item[(3)] $\dim\sV_d\geq2,\dim\sV_{d'}=1\Ra \text{ the eigenvalue of } h \text{ on } \sV_d \text{ is } \alpha(h)$: If $\zeta$ is the eigenvalue of $h$ on $\sV_d$ with $\zeta\neq\alpha(h)$, then $\alpha(h)^{-1}\zeta\neq1$ would be an eigenvalue on $V$ of multiplicity $\dim\sV_d\geq2$.
\item[(4)] $\dim\sV_{d'}=2\Ra\dim\sV_d=2$: Let $\dim\sV_d>2$. There are two eigenvalues $\zeta\neq\zeta'$ of $h$ on $\sV_d$ of multiplicity greater or equal to $2$. Hence we have on $V$ the eigenvalues $\alpha(h)\inv\zeta$ and $\alpha(g)\inv\zeta'$ of the same multiplicity.
\item[(5)] The case $\dim\sV_{d'}=\dim\sV_d=2$ cannot occur: Let $\dim\sV_{d'}=\dim\sV_d=2$ with eigenvalues $\zeta,\zeta'$ of $h$ on $\sV_d$. Without loss of generality we can assume that $\zeta=\alpha(h)$. If not we would have eigenvalues $\alpha(h)\inv\zeta\neq1$ and $\alpha(h)\inv\zeta'\neq1$ on $V$. There can be no other summand $\sV_{d_1}$ in the decomposition of $\Lambda_\qnf{D}$, as this summand would give an eigenvalue $\neq1$ (the dimension of $\sV_{d'}$ has to be $1$, but as $\zeta=\alpha(h)$ and $\zeta$ is a primitive $d'$th root of unity, this cannot happen). There are  two eigenvalues of $h$ on $\VV_r^\omega$ (because of $\dim\sV_{d'}=2$) which we will call $\alpha(h)$ and $\zeta''$ with multiplicity $\frac{\dim\sV_r}{2}$ (the denominator is $\dim\sV_{d'}$). Therefore the multiplicity of the eigenvalues have to be $1$, because $\alpha(h)\inv\zeta''\neq1$ is an eigenvalue on $V$. But then we will have two eigenvalues $\neq1$ on $V$ (namely $\alpha(h)\inv\zeta'$ and $\alpha(h)\inv\zeta''$).
\end{enumerate}
Hence there follows $\dim\sV_{d'}=1$. Now we want to study $\sV_r$. Let $r':=\frac{r}{(k,r)}$. We claim that $\dim\sV_{r'}=1$. Suppose $\dim\sV_{r'}\geq2$.
\begin{enumerate}
  \item[(6)] $\dim\sV_{r'}\leq2$: Assume that $\dim\sV_{r'}>2$, i.e. $h$ has on $\sV_r^\omega$ at least three distinct eigenvalues $\alpha(h),\zeta,\zeta'$, which will give rise to eigenvalues $\alpha(h)\inv\zeta\neq1$ and $\alpha(h)\inv\zeta'\neq1$ on $V$.
\item[(7)] $\dim\sV_{r'}=2\Ra n=1$: We know  $\dim\sV_{d'}=1$ from above. Let $\zeta$ be the eigenvalue of $h$ on $\sV_d$ of multiplicity $\dim\sV_d$. Clearly $\zeta\neq\alpha(h)$, because of  dimension reasons. So we get the eigenvalue $\alpha(h)\inv\zeta$ on $V$, and hence $\Lambda_\qnf{D}=\sV_r^\omega$ and $\rk\Lambda=2$.
\end{enumerate}
By the assumption $n\geq2$ we get $\dim\sV_{r'}=1$. Putting this all together we get as a $h$-module
\begin{eqnarray*}
  \Lambda_\qnf{D}&\isom&\sV_r^\omega\oplus\bigoplus_i\sV_{d_i},
\end{eqnarray*}
where the eigenvalues of $h=g^k$ on
\begin{enumerate}
  \item[(a)] $\sV_r^\omega$ are primitive $r'$th roots of unity ($\dim\sV_{r'}=1$) of multiplicity $\dim\sV_r$.
\item[(b)] $\sV_{d_i}$ are primitive $d_i'$th roots of unity ($\dim\sV_{d_i'}=1$) of multiplicity $\dim\sV_{d_i}$.
\end{enumerate}
\end{proof}
\begin{corollary}\label{CorollaryQRareinducedbyhwhichlookas}
 The quasi-reflections on $V$ are induced by elements $h\in U(\Lambda)$, such that
\begin{enumerate}
  \item[(i)] $\pm h$ acts as a reflection on $\Lambda_\CC$, if $D<D_0$ or $D=-2$,
\item[(ii)] $h^4\sim I$, if $D=-1$,
\item[(iii)] $h^6\sim I$, if $D=-3$.
\end{enumerate}
\end{corollary}
\begin{proof}
  One has to check all possibilities for $\alpha(h)$.
\end{proof}
It is enough to investigate quotients $V/\left< g\right>$ since $V/G$ has canonical singularities if $V/\left< g\right>$ has canonical singularities for all $g\in G$. This was shown by \cite[Proof of Lemma 2.14]{MR2336040}.

Assume for the quasi-reflection $h=g^k$ that $k>1$ is minimal with this property. Then the quotient $V':=V/\left< h \right>$ is smooth. Let $h$ be of order $l$, so $g$ has order $lk$. Now consider the eigenvalues $\zeta^{a_1},\dots,\zeta^{a_n}$ of $g$ on $V$, where $\zeta$ denotes a primitive $lk$th root of unity. Now we consider the action of the group $\latt{g}/\latt{h}$ on $V'$. Using analogous arguments as before we want to describe the action of $g^f\latt{h}\in\latt{g}/\latt{h}$on $V'$. The differential of $g^f\latt{h}$ on $V'$ has eigenvalues $\zeta^{fa_1},\dots,\zeta^{fa_{n-1}},\zeta^{lfa_n}$.  Now we have to modify the Reid-Tai sum as
 \begin{eqnarray}
  \Sigma'(g^f):=\left\{\dfrac{fa_n}{k}\right\}+\sum_{i=1}^{n-1}\left\{\dfrac{fa_i}{lk}\right\}.\label{DefinitionSigmaprime}
\end{eqnarray}
\begin{lemma}\label{LemmacasesinteriorSigmaandSigmaprimegeq1}
   The   variety $\Gamma\backslash\CHn$ has canonical singularities, if
\begin{itemize}
  \item[(i)] $\Sigma(g)\geq1$ for all $g\in\Gamma$ no power of which is a quasi-reflection, and
\item[(ii)] $\Sigma'(g^f)\geq1$ for $1\leq f<k$, where $h=g^k$ is a quasi-reflection.
\end{itemize}
\end{lemma}
\begin{proof}
  \cite[Lemma 2.14]{MR2336040}.
\end{proof}
\begin{proposition}\label{PropositionSigmaprimegeq1interior}
  Let $h=g^k$ be as above, $D<D_0$ and $n\geq12$. Then $  \Sigma'(g^f)\geq1$ for $1\leq f<k$.
\end{proposition}
\begin{proof}
    We know from the former results that all eigenvalues on  $\sV_r^\omega$ are $\alpha(h)$, where 
\begin{eqnarray*}
  \alpha(h)&=&
\begin{cases}
              \pm1,& D<D_0\ \text{and}\ D=-2,\\
\pm1, \zeta_4,&D=-1,\\
\pm1,\zeta_3,\zeta_6,&D=-3.
\end{cases}
\end{eqnarray*}
By a detailed analysis of the decomposition of $\Lambda_\CC$ into $\qnf{D}$-irreducible pieces there is exactly one eigenvalue on $\Lambda_\CC$ that is $\lambda\not=\alpha(h)$, since only one eigenvalue on $V$ is not $1$. This eigenvalue $\lambda$ will appear on one $\sV_d$. As all eigenvalues of $g$ on $\sV_d$ are primitive $d$th roots of unity they all have the same order. We know that $\lambda$ must have multiplicity $1$ on $\Lambda_\CC$ so $\dim\sV_d=1$. This implies
\begin{eqnarray}
  d=\begin{cases}
      1,2,\\1,2,4,\\1,2,3,6.
    \end{cases}
\end{eqnarray}
Denote by $v$ the eigenvector of $g$ corresponding to the eigenvalue $\zeta^{a_n}$. Then $v$ clearly comes from $\sV_d$ and therefore $\left< v\right>=\Hom(\WW,\VV_d)$.
 If $\delta$ is the primitive generator of $\sV_d\cap\Lambda$ then $h(\delta,\delta)>0$, since $\sV_d\subset W_\qnf{D}^\perp$, where $W_\qnf{D}\otimes_\qnf{D}\CC\isom\WW$ and $W_\qnf{D}$ is a $\qnf{D}$-vector space. The form $h(\cdot,\cdot)$ is negative definite on $\WW$ as shown in the proof of Lemma \ref{LemmaScapT=0}. If we define the sublattice $\Lambda'\subset\Lambda$ as $\Lambda':=\delta^\perp$, this lattice has signature $(n-1,1)$. Now $\left<g\right>/\left<h\right>$ acts on $\Lambda'$ as a subgroup of $U(\Lambda')$. Therefore $$\Sigma'(g^f)=\left\{\dfrac{f{a_n}}{k}\right\}+\Sigma(g^f\latt{h})$$ 
and $g^f\latt{h}\in U(\Lambda')$. Analogously to the proof of \cite[Proposition 2.15]{MR2336040} we can give the following argument: we claim that $g^f\latt{h}$ is not a quasi-reflection on $\Lambda'$. If it were, the eigenvalues of $g^f$ on $\Lambda'$ are as in Corollary \ref{CorollaryQRareinducedbyhwhichlookas}. Thus the order of the eigenvalue on $\sV_d$ is
 \begin{eqnarray*}
  d=\begin{cases}
      1,2,\\1,2,4,\\1,2,3,6.
    \end{cases}
\end{eqnarray*}
So $\ord g^f$ divides $l$, and therefore $g^f\in\latt{h}$. Hence the group $\latt{g}/\latt{h}$ has no quasi-reflections and we apply Theorem \ref{TheoremforgnoQRandngeq11} for $n-1\geq11$.
\end{proof}
\begin{theorem}\label{Theoremngeq12interiorcansings}
  Let $n\geq 12$ and $D<D_0$. Then $\Gamma\backslash\CHn$ has canonical singularities.
\end{theorem}
\begin{proof}
  This follows directly from Lemma \ref{LemmacasesinteriorSigmaandSigmaprimegeq1}, Theorem \ref{TheoremforgnoQRandngeq11} and Proposition \ref{PropositionSigmaprimegeq1interior}.
\end{proof}

\section{The boundary}\label{SectionBoundary}
Now we want to state a result on the singularities of the toroidal compactification $(\Gamma\backslash\CHn)^*$ of the quasi-projective variety $\Gamma\backslash\CHn$. For a complex matrix $A$ we write $\hermitian{A}$ instead of $^T{\overline{A}}$.

Therefore we consider isotropic subspaces $E_\qnf{D}$ with respect to the form $h(\cdot,\cdot)$. As the form is of signature $(n,1)$ they are $1$-dimesional. To each isotropic subspace there corresponds a $0$-dimensional boundary component or cusp $F$. First we choose a basis such that $h(\cdot,\cdot)$ is given by the matrix $Q'$, i.e. it can be written in the form
$$h(x,y)=\hermitian{y}Q'x.$$
The next goal is to find a basis such that the form behaves well in later calculations.
\begin{lemma}
  There exists a basis $b_1,\dots,b_{n+1}$ of $\Lambda_\qnf{D}$, such that 
\begin{enumerate}
  \item[(i)] $b_1$ is a basis of $E_\qnf{D}$ and $b_1,\dots,b_n$ is a basis of $E_\qnf{D}^\perp$,
\item[(ii)] the hermitian form is written with respect to this basis as 
\begin{eqnarray}
  Q:=(h(b_i,b_j))_{1\leq i,j\leq n+1}=\left(
\begin{array}{c|c|c}
  0 &0&a\\
\hline
0&B&0\\
\hline
\bar{a}& 0 & 0
\end{array}
\right),
\end{eqnarray}
where $a\in\qnf{D}$ and $B=\hermitian{B}$.
\end{enumerate}
\end{lemma}
\begin{proof}
First we can show that $Q'$ has to be of the form $ Q'=\left(
\begin{array}{c|c|c}
  0 &0&a\\
\hline
0&B&c\\
\hline
\bar{a}& \hermitian{c} & d
\end{array}
\right)$.
The upper zeroes in the matrix $Q'$ and $Q'=\hermitian{Q'}$ directly follow from the fact that $h(E_\qnf{D},e)=0$ for all $e\in E_\qnf{D}^\perp$ and $h(x,y)=\overline{h(y,x)}$. The rest is similar to \cite[Proof of Lemma 2.24]{MR2336040}. The matrix $B$ represents the hermitian form $h$ on $E_\qnf{D}^\perp/E_\qnf{D}$ and is therefore invertible. Thus one can define 
\begin{eqnarray*}
  N:=\left(
\begin{array}{c|c|c}
  1 &0&r'\\
\hline
0&I_{n-1}&r\\
\hline
0& 0 & 1
\end{array}
\right),
\end{eqnarray*}
where $r:=-B\inv c\in\qnf{D}^{n-1}$. Choose $r'$ such that it satisfies the equation
\begin{eqnarray*}
  d-\hermitian{c}B\inv c+\overline{r'}a+\overline{a}r'=0.
\end{eqnarray*}
This is possible as the first two summands are real by definition and the other two are the complex conjugate of each other and therefore their sum is real.
Now
\begin{eqnarray}
  \hermitian{N}Q'N&=&\left(
\begin{array}{c|c|c}
  0 &0&a\\
\hline
0&B&Br+c\\
\hline
\overline{a}& \hermitian{r}B+\hermitian{c} & \delta
\end{array}
\right),
\end{eqnarray}
with $\delta:=\overline{a}r'+(\hermitian{r}B+\hermitian{c})r+\overline{r'}a+\hermitian{r}c+d$. But 
$$Br+c=B(-B\inv c)+c=0.$$
Because of the definition of $r$ and $r'$ we achieve
\begin{eqnarray}
  \delta&=&\overline{a}r'+\hermitian{(-B\inv c)}B(-B\inv c)+\hermitian{c}(-B\inv c)+\overline{r'}a+\hermitian{(-B\inv c)}c+d\notag\\
&=&\underbrace{\overline{a}r'+\overline{r'}a-\hermitian{c}\hermitian{(B\inv)}c+d}_{=0} +\underbrace{\hermitian{c}\hermitian{(B\inv)}B B\inv c-\hermitian{c}B\inv c}_{=0}\notag\\
&=&0.\notag
\end{eqnarray}
Note that $\hermitian{(B\inv)}=B\inv$.
Altogether this gives the result.
\end{proof}

To continue in the compactification we follow \cite{MR0457437}. Therefore we first calculate the stabiliser subgroup.
\begin{lemma}\label{LemmacalculationN(F)}
  Let $N(F)\subset\Gamma_\RR$ be the stabiliser subgroup corresponding to the cusp $F$. Then 
\begin{eqnarray}
  N(F)&=&\left\{g=\left(
\begin{array}{c|c|c}
  u&v&w\\
\hline
0&X&y\\
\hline
0&0 & z
\end{array}
\right);\begin{array}{c} z\overline{u}=1,\ \hermitian{X}BX=B,\\ \hermitian{X}By+\hermitian{v}az=0,\\ \hermitian{y}By+\overline{z}\overline{a}w+za\overline{w}=0\end{array}\right\}.
\end{eqnarray}
\end{lemma}
\begin{proof}
  This follows directly when we study the $g\in\Gamma_\RR$ that satisfy $gb_1=b_1$, and drop all $g$ that do not respect the form defined by $Q$.
\end{proof}
\begin{lemma}\label{LemmaW(F)}
  The unipotent radical is
  \begin{eqnarray}
  W(F)&=&\left\{g=\left(
\begin{array}{c|c|c}
  1&v&w\\
\hline
0&I_{n-1}&y\\
\hline
0&0 & 1
\end{array}
\right);\begin{array}{c} By+\hermitian{v}a=0,\\ \hermitian{y}By+\overline{a}w+a\overline{w}=0\end{array}\right\}
\end{eqnarray}
\end{lemma}
\begin{proof}
  The group $W(F)$ is by definition the subgroup of $N(F)$ consisting of all unipotent elements. Therefore an element $g\in W(F)$ has to be of the form
$$g=\left(
\begin{array}{c|c|c}
  1&v&w\\
\hline
0&X&y\\
\hline
0&0 & 1
\end{array}
\right),$$
where $X=I_{n-1}+T$ with $T$ strict upper triangular. So it remains to show that $T=0$. As $B$ is definite and $X$ is unipotent the statement follows by induction on $n$.
\end{proof}
\begin{lemma}\label{LemmaU(F)}
   The centre of $W(F)$ is then given by the group
\begin{eqnarray}
  U(F)&=&\left\{g=\left(
\begin{array}{c|c|c}
  1&0&iax\\
\hline
0&I_{n-1}&0\\
\hline
0&0 & 1
\end{array}
\right);\ x\in \RR\right\}
\end{eqnarray}
\end{lemma}
\begin{proof}
  The first condition of $W(F)$ gives
$
  v=\hermitian{\left(-\frac{1}{a}By\right)}.
$
Now we will use that $U(F)$ is the centre of $W(F)$, i.e. $$\Centre(W(F))=\left\{g\in W(F);\ gg'=g'g\ \text{for all}\ g'\in W(F)\right\}.$$
These products for $g, g'\in W(F)$ lead to
$$\begin{array}{crcl}
  &vy'&=&v'y\\
\Longra&\hermitian{\left(-\frac{1}{a}By\right)y'}&=&\hermitian{\left(-\frac{1}{a}By'\right)y}\\
\Longlra&\hermitian{y}By'-\hermitian{\left(\hermitian{y}B y'\right)}&=&0.
\end{array}$$
Clearly the last equivalence implies that $\hermitian{y}By'\in\RR$ for every $y'$.
The matrix $B$ has full rank as it is invertible and thus $B\cdot\CC^{n-1}=\CC^{n-1}.$
 Therefore set $z':=By'\in\CC^{n-1}$. Now we rephrase the property from above as 
\begin{eqnarray}
\hermitian{y}z'\ \text{is real for all}\ z'\in\CC^{n-1}.\label{conditionyzprimeisreal}
\end{eqnarray}
As this is true for all vectors we can choose $z'$ to be $$z'=\transpose{(0,\dots,0,1,0,\dots,0)},$$ 
where the only coordinate not equal to $0$ is the $j$th. For this choice in (\ref{conditionyzprimeisreal}) only the $j$th coordinate of $y$ remains and therefore $\overline{y}_j\in\RR$.
Now let $$z'=\transpose{(0,\dots,0,\sqrt{D},0,\dots,0)}.$$ 
Then  (\ref{conditionyzprimeisreal}) becomes $\overline{y}_j\cdot\sqrt{D}\in\RR$, and as $D<0$ this means $y_j\in i\RR$. Hence 
$$y_j\in \RR\cap i\RR=\{0\},\ \text{because}\ z'\ \text{varies in}\ \CC^{n-1}.$$
As $j$ is chosen arbitrary we can deduce that this is true for every entry, i.e. $y=0$ and therefore also $v=0$.
So we have to study the remaining condition $\overline{a}w+\overline{w}a=0$. 
We want to describe $w$ more specifically, i.e. in terms of $a$. For this we write $w=c+id$ and $a=e+if$. So we get
\begin{eqnarray*}
  \overline{a}w+\overline{w}a=2(ec+df)=0.
\end{eqnarray*}
Assuming $e\neq0$  this implies $c=-d\frac{f}{e}$ and for this reason $w=-d\frac{f}{e}+id$, $d\in\RR$. Therefore
$w\in\RR\left(-\frac{f}{e}+i\right)=i\RR(e+if)=ia\RR$.
The case $f\neq0$ is similar.
\end{proof}
\begin{lemma}\label{LemmaU(F)ZZisomZZ}
  $U(F)_\ZZ=U(F)\cap\Gamma\isom\ZZ$.
\end{lemma}
\begin{proof}
  As $\Gamma\subset\GL(n+1,\sO)$ it is clear that $iax\in\sO$. Also note that $x\in\RR$. First consider the case $D\equiv 2,3\mod4$. Therefore 
\begin{eqnarray}
  iax=c+d\sqrt{D}\ \text{for some}\ c,d\in\ZZ.\label{Darstellung_iaxinU(F)ZZforD23}
\end{eqnarray}
Additionally we know that $a\in\qnf{D}$ and hence $a=e+f\sqrt{D}$ for some $e,f\in\QQ$. Thus we can write equation (\ref{Darstellung_iaxinU(F)ZZforD23}) as 
\begin{eqnarray*}
 & i(e+f\sqrt{D})x=c+d\sqrt{D}\\
\Lra & f\sqrt{-D}x+iex=c+d\sqrt{D}.
\end{eqnarray*}
Therefore $fx\sqrt{-D}\in\ZZ\ \text{and}\ iex\in\ZZ\sqrt{D}$,
so we get $x\in\frac{1}{f(-D)}\ZZ\sqrt{-D}\cap \frac{1}{e}\ZZ\sqrt{-D}$.
As $e,f\in\QQ$ choose $e=\frac{p}{q}, f=\frac{r}{s}$ coprime and set $\tilde{x}\sqrt{-D}=x$, hence 
$$\tilde{x}\in\frac{s}{r(-D)}\ZZ\cap \frac{q}{p}\ZZ.$$
We claim 
\begin{eqnarray*}
  \frac{s}{rD'}\ZZ\cap \frac{q}{p}\ZZ=\frac{\lcm(sp,rD'q)}{rD'p}\ZZ,
\end{eqnarray*}
 and define $D':=-D,\ c_1sp:=\lcm(sp,rD'q),\ c_2rD'q:=\lcm(sp,rD'q)$.

We will first prove `$\supset$'. Let $\eta\in
\dfrac{\lcm(sp,rD'q)}{rD'p}\ZZ$. Thus we can write with $c_1,c_2$ defined as above and $c\in\ZZ$: 
\begin{eqnarray*}
  \eta=\dfrac{\lcm(sp,rD'q)}{rD'p}c&=&\dfrac{c_1sp}{rD'p}c=\dfrac{c_2rD'q}{rD'p}c\\
&=&\dfrac{c_1s}{rD'}c=\dfrac{c_2q}{p}c.
\end{eqnarray*}
We have to find $a=a(c),b=b(c)\in\ZZ$, such that we can write $\eta$ in the form $\dfrac{s}{rD'}a,\dfrac{q}{p}b$.
Now let $a:=c_1c,\ b:=c_2c$ and with this choice $\eta$ lies in $\dfrac{s}{rD'}\ZZ$ and in $\frac{q}{p}\ZZ$ and therefore in
$\frac{s}{rD'}\ZZ\cap\frac{q}{p}\ZZ$.

   Now we deal with `$\subset$'. Choose $\eta\in\frac{s}{rD'}\ZZ\cap\frac{q}{p}\ZZ$, i.e. there exist $a,b\in\ZZ$ with 
\begin{eqnarray}
       \eta=\dfrac{s}{rD'}a=\dfrac{q}{p}b.\label{equationetaausdurchschnitt}                       
\end{eqnarray}
We have to show that there exists a $c(a,b)=c\in\ZZ$ with $\eta=\dfrac{\lcm(sp,rD'q)}{rD'p}c$. Now let $c:=\dfrac{b}{c_2}=\dfrac{a}{c_1}$. Writing the first part of (\ref{equationetaausdurchschnitt}) with this choice of $c$ leads to
\begin{eqnarray*}
  \eta=\dfrac{s}{rD'}cc_1=\dfrac{spc_1}{rD'p}c=\dfrac{\lcm(sp,rD'q)}{rD'p}c.
\end{eqnarray*}
The other case is analogous. So it remains to show that this choice of $c$ lead to integers. This can be seen in the following way: By (\ref{equationetaausdurchschnitt}) we get $sap=qbrD'$, and multiplying this by $c_1c_2$ gives 
$$
\begin{array}{rrcl}
   &sapc_1c_2&=&qbrD'c_1c_2\\
\Longlra& ac_2\lcm(sp,rD'q)&=&bc_1\lcm(sp,rD'q)\\
\Longlra &ac_2&=&bc_1.
\end{array}$$
We know that $c_1$ and $c_2$ are coprime because they are defined by the lowest common multiple. From this  and the equation above it follows that $c_1$ divides $a$ and $c_2$ divides $b$. Thus $c\in\ZZ$ as required.
The case $D\equiv 1\mod4$ is similar.
\end{proof}
This leads to the construction of a toroidal compactification. We have a $\ZZ$-lattice of rank $1$ in the complex vector space $U(F)_\CC=U(F)\otimes_\ZZ\CC$. To give a local compactification of $\Gamma\backslash\CHn$ we will choose coordinates on $\CHn$, namely $(t_1:\dots:t_{n+1})$. By the definition of $\CHn$ we can assume that $t_{n+1}=1$. We will compactify $\CHn$ locally in the direction of the cusp $F$. Therefore we will denote the partial quotient by
$$\CHn(F):=\CHn/U(F)_\ZZ.$$
By standard calculations this can be identified with
\begin{eqnarray}
  \CHn(F)\isom \CC^*\times\CC^{n-1}.
\end{eqnarray}
For this identification we introduce new variables $\alpha$ and $\underline{w}=(w_2,\dots,w_n)$:
\begin{eqnarray*}
  t_1&\mapsto& \alpha\in\CC^*,\\
t_i &\mapsto& w_i\in\CC,\ 2\leq i\leq n.
\end{eqnarray*}
We need an explicit description of the action of the  group $N(F)$ on $\CHn(F)$. 
\begin{lemma}\label{LemmaN(F)actsonCHn}
   If 
\begin{eqnarray*}
  g=\left(
\begin{array}{c|c|c}
  u&v&w\\
\hline
0&X&y\\
\hline
0&0 & z
\end{array}
\right)\in N(F),
\end{eqnarray*}
then $g$ acts on $\CHn$ by
\begin{eqnarray}
  \alpha&\mapsto&\frac{1}{z}\left(\frac{\alpha}{\overline{z}}+v\underline{w}+w\right),\notag\\
\underline{w}&\mapsto&\frac{1}{z}\left(X\underline{w}+y\right).\notag
\end{eqnarray}
\end{lemma}
\begin{proof}
  This easily follows from the computation
  \begin{eqnarray*}
    \left(
\begin{array}{c|c|c}
  u&v&w\\
\hline
0&X&y\\
\hline
0&0 & z
\end{array}
\right)
\left(\begin{array}{c}
  \alpha\\ \underline{w}\\1
\end{array}\right)=
\left(\begin{array}{c}
  u\alpha+v\underline{w}+w\\
X\underline{w}+y\\z
\end{array}\right)=
\left(\begin{array}{c}
 \dfrac{u\alpha+v\underline{w}+w}{z}\\
\dfrac{X\underline{w}+y}{z}\\1
\end{array}\right).
  \end{eqnarray*}
and the property $u=(\overline{z})\inv$ from Lemma \ref{LemmacalculationN(F)}.
\end{proof}
Define the algebraic torus $T$ as
\begin{eqnarray*}
  T:=U(F)_\CC/U(F)_\ZZ\isom\CC^*.
\end{eqnarray*}
We define a variable $\theta$ on  $T$ by
\begin{eqnarray*}
 \theta:= \exp_a(\alpha):=\left\{
\begin{array}{cl}
       e^{\frac{2\pi rD'p}{a\lcm(sp,rD'q)\sqrt{-D}}\alpha}=:e^{\frac{2\pi i}{\sigma}\alpha}, &D\equiv2,3\mod4,\\
e^{\frac{4\pi rD'p}{a\lcm(sp,rD'q)\sqrt{-D}}\alpha}=:e^{\frac{2\pi i}{\sigma}\alpha},&D\equiv1\mod4,
\end{array}\right.
\end{eqnarray*}
where we use the same notation as in the proof of Lemma \ref{LemmaU(F)ZZisomZZ}. This variable has to be invariant under the action of $U(F)_\ZZ$, i.e. $\alpha\mapsto\alpha+iax=\alpha+\sigma b$ for a $b\in\ZZ$ and $\sigma$ as above.
Let $g\in G(F)=N(F)_\ZZ/U(F)_\ZZ$ and suppose that $g$ has order $m>1$. We will also write $g$ if we think of $g$ as an element of $N(F)$.
If we want to compactify $\Gamma\backslash\CHn$ locally around the cusp $F$ that means that we allow $\theta=0$. So we add $\{0\}\times\CC^{n-1}$ to the boundary modulo the action of $G(F)$ which extends uniquely to the boundary.
We now want to apply the techniques from section \ref{SectionInterior} to the boundary. Suppose now that $g$ fixes the boundary point $(0,\underline{w}_0)$ for some $\underline{w}_0\in\CC^{n-1}$. Let $\zeta^{a_i}$ be the eigenvalues of the action of $g$ on the tangent space, where $\zeta$ denotes a primitive $m$th root of unity. Thus we can, as before, define the Reid-Tai sum $\Sigma(g)$.
\begin{proposition}\label{PropositiongnoQRfixingboundarySigmag_geq1}
  Suppose no power of $g$ acts as a quasi-reflection at the boundary point $(0,\underline{w}_0)$ and $D<D_0$. Then $\Sigma(g)\geq1$.
\end{proposition}
\begin{proof}
  As $D<D_0$ we can assume $z=\pm1$, because $z$ is invertible in $\sO$ as $z\overline{u}=1$ by Lemma \ref{LemmacalculationN(F)}.
Now we have to determine the action of $g$ on the tangent space. This is for obvious reasons given by the matrix
\begin{eqnarray*}
 J= \left(\begin{array}{cc}
\exp_a(\pm(v\underline{w}_0+w)) & 0\\
\ast & \pm X
\end{array}\right).
\end{eqnarray*}
We denote the order of $X$ by $m_X$ and investigate the decomposititon of the representation $X$. As before the representation decomposes into a direct sum of $\sV_d$'s. We have to distinguish two cases. 
First assume that $m_X>2$. In this case we are in the situation of Lemma \ref{LemmagnoQRR=1,2}, as we are in case $D<D_0$ and the only irreducible $1$-dimensional representations are $V_1$ and $V_2$. So by the lemma we get $\Sigma(g)\geq1$.
 Now let $m_X=1$ or $m_X=2$. The action of $-1\in\Gamma$ is trivial and so we can get $z=1$ by replacing $g$ by $-g$. Assume $m_X=1$ and hence $X=I$. As the element $g$ fixes the boundary point $(0,\underline{w}_0)$ we get $y=0$ from Lemma \ref{LemmaN(F)actsonCHn} and then by the group relations of Lemma \ref{LemmacalculationN(F)} we have $v=0$ since $\hermitian{v}a=0$. So the element $g$ has to have the form
$$
g=\left(
\begin{array}{c|c|c}
  1&0&w\\
\hline
0&I&0\\
\hline
0&0 & 1
\end{array}
\right),
$$
and hence $g\in U(F)_\ZZ$. This implies that $g\in N(F)_\ZZ/ U(F)_\ZZ$ is the identity.
Finally we have to check the case $m_X=2$. So $g^2\in U(F)_\ZZ$, and therefore we get the following relations, where  $\sigma$ is as before:
\begin{eqnarray}
  v+vX&=&0,\label{equationv+vX=0}\\
Xy+y&=&0,\notag\\
2w+vy&\equiv&0\mod \sigma.\label{congruence2w+vycong0}
\end{eqnarray}
We only consider the case $D\equiv2,3\mod4$ as the case $D\equiv1\mod4$ is analogous.
Define $t:=v\underline{w}_0+w$ which is the argument of the  exponential map  in the  matrix $J$. We want to show $2t\equiv0\mod\sigma\ZZ$ as this implies
$\exp_a(t)=\pm1$.
We will now use  $\underline{w}_0=X\underline{w}_0+y$ as $g$ fixes the boundary point and the relations (\ref{equationv+vX=0}), (\ref{congruence2w+vycong0}). Hence we get
\begin{eqnarray*}
  2t=2v\underline{w}_0+2w&\equiv&2v\underline{w}_0-vy\\
&=&v\underline{w}_0+v\underline{w}_0-vy=v\underline{w}_0+v(\underline{w}_0-y)\\
&=&v\underline{w}_0+vX\underline{w}_0=v(I+X)\underline{w}_0\\
&\equiv&0\mod \sigma.
\end{eqnarray*}
Therefore all the eigenvalues on the tangent space are $\pm1$, as $X$ has order $2$ and $\exp_a(t)=\pm1$ for $t$  as above.
So there are two possibilities: all but one of the eigenvalues are $+1$, so $g$ acts as a reflection (in this case all quasi-reflections have order $2$), or there are at least two eigenvalues $-1$ and the remaining are $+1$, so we will have $\Sigma(g)\geq1$.
\end{proof}
\begin{corollary}\label{Corollarynotfixingboundarydivisor}
  At the boundary there are no divisors over a dimension $0$ cusp $F$ that are fixed by a non-trivial element of $N(F)_\ZZ/U(F)_\ZZ$ in the case $D<D_0$.
\end{corollary}
\begin{proof}
  Each divisor at the boundary has $\theta=0$. The only elements fixing a divisor are the quasi-reflections. The variable $\theta$ corresponds to the entry
$\exp_a(\pm(v\underline{w}_0+w))$ from the induced action on the tagent space. From the proof of Proposition \ref{PropositiongnoQRfixingboundarySigmag_geq1} each matrix $X$ belonging to a quasi-reflection has order greater than $1$. Thus no divisor $\theta=0$ is fixed.
\end{proof}
Finally we have to consider quasi-reflections at the boundary, which will be done as in section \ref{SectionInterior}. Therefore define $\Sigma'(g)$ for $g\in G(F)$ as in (\ref{DefinitionSigmaprime}). 
\begin{proposition}\label{Propositionh=gkQRSigmaprimegeq1}
  Let $g\in G(F)$ be such that $h=g^k$ is a quasi-reflection. Assume that $n\geq13$ and $D<D_0$. Then $\Sigma'(g^f)\geq1$ for every $1\leq f< k$.
\end{proposition}
\begin{proof}
  The proof is similar to the proof of \cite[Proposition 2.30]{MR2336040}.
We will again study the action of $h$ on the tangent space. If $\exp_a(t)$ is the  eigenvalue not equal to $1$, then  $X^f$ contributes at least $1$ to $\Sigma'(g^f)$.
  Now denote this unique eigenvalue of $h$ on the tangent space by $\zeta\not=1$. Let $\nu$ be the exceptional eigenvector of of $h$ with the property $h(\nu)=\zeta\cdot \nu$. Consider the decomposition of $X$ as a $g$-module and assume that  $\nu$ occurs in the representation $\sV_d$. The dimension of $\sV_d$ has to be $1$ as otherwise it would contribute another eigenvalue not equal to $1$. Now we study the $g$-module 
$$E_\qnf{D}^\perp/(E_\qnf{D}+\qnf{D}\nu),$$ which is $(n-2)$-dimensional. We can refer to Theorem \ref{TheoremforgnoQRandngeq11} as long as $D<D_0$.
So $\Sigma(g)\geq1$ if $n-2\geq11$  and thus $\Sigma'(g)\geq1$.
\end{proof}
\begin{theorem}
  Let $n\geq13 $ and $D<D_0$. Then the toroidal compactification $(\CHn/\Gamma)^*$ of $\CHn/\Gamma$ has canonical singularities. Furthermore, there are no fixed divisors in the boundary.
\end{theorem}
\begin{proof}
  This is a consequence of Theorem \ref{Theoremngeq12interiorcansings}, Proposition \ref{PropositiongnoQRfixingboundarySigmag_geq1}, Corollary \ref{Corollarynotfixingboundarydivisor} and Proposition \ref{Propositionh=gkQRSigmaprimegeq1}.
\end{proof}


\end{document}